\documentclass[a4paper,11pt,leqno]{article}
\usepackage{amsmath,amsfonts,amsthm,amssymb}
\usepackage[left=3cm, right=3cm, top=3cm, bottom=3.5cm]{geometry}
\usepackage{graphicx}
\numberwithin{equation}{section}

\theoremstyle{plain}
\newtheorem{theorem}{Theorem}[section]

\newtheorem{lemma}[theorem]{Lemma}
\newtheorem{proposition}[theorem]{Proposition}
\newtheorem{corollary}[theorem]{Corollary}

\theoremstyle{remark}
\newtheorem*{remark}{Remark}

\theoremstyle{definition}
\newtheorem*{definition}{Definition}

\newcommand{\R}{\mathbb{R}}

\newcommand{\nat}{{{\rm I} \kern -.15em {\rm N} }}

\begin{document}

\title{On the asymptotic behavior of symmetric solutions of the Allen-Cahn equation in unbounded domains in $\R^2$\\}
\author{  Giorgio Fusco\footnote{ Universit\`a degli Studi di L'Aquila, Via Vetoio, 67010 L'Aquila, Italy; e-mail:{\texttt{fusco@univaq.it}}},\ \ Francesco Leonetti\footnote{Dipartimento di Ingegneria e
 Scienze dell'Informazione e Matematica, Universit\`a degli Studi di L'Aquila, Via Vetoio, 67010 L'Aquila, Italy; e-mail:{\texttt{leonetti@univaq.it}}} \ \ and Cristina Pignotti\footnote{Dipartimento di  Ingegneria e
 Scienze dell'Informazione e Matematica, Universit\`a degli Studi di L'Aquila, Via Vetoio, 67010 L'Aquila, Italy; e-mail:{\texttt{pignotti@univaq.it}}} }
\date{}
\maketitle

\begin{abstract}
We consider a Dirichlet problem for the Allen--Cahn equation in a smooth, bounded or unbounded, domain $\Omega\subset\R^n.$
Under suitable assumptions, we prove an existence result and a uniform exponential estimate for symmetric solutions. In dimension $n=2$ an additional asymptotic result is obtained.
These results are based on a pointwise estimate obtained for local minimizers of the Allen--Cahn energy.
\end{abstract}

\section{Introduction}
We consider the Allen-Cahn equation
\begin{eqnarray}\label{equation}
\left\{\begin{array}{l}\Delta u = W^{\prime}(u), \quad x\in\Omega,\\
 u=g, \hspace{1.5 cm} x\in \partial\Omega.
\end{array}\right.
\end{eqnarray}
where $\Omega\subset\R^n$ is a bounded or unbounded domain, $g:\partial\Omega\rightarrow\R$ is continuous and bounded and $W:\R\rightarrow\R$ is a $C^3$ potential.

We are interested in symmetric solutions:
 \[u(\hat{x})=-u(x),\;\text{ for }\;x\in\Omega\]
 where for $z\in\R^d$ we let $\hat{z}=(-z_1, z_2,\dots,z_d)$ the reflection in the plane $z_1=0$.
We assume:
\begin{description}
\item[h$_1-$] $W:\R\rightarrow\R$ is an even function:
\begin{eqnarray}
W(-u)=W(u),\; \text{ for }\;u\in\R,
\end{eqnarray}
which has a unique non-degenerate positive minimizer:
\begin{eqnarray}\label{potential}
&&0=W(1)<W(u),\; \text{ for }\;u\geq 0,\\\nonumber
&&W^{\prime\prime}(1)>0.
\end{eqnarray}

\item[h$_2-$] There is $M>0$ such that
\begin{eqnarray}\label{m-exists}
W^\prime(u)\geq 0,\;\text{ for }\;u\geq M.
\end{eqnarray}

\item[h$_3-$] $\Omega\subset\R^n$ is a domain with nonempty boundary which is symmetric:
\begin{eqnarray}\label{omega-symmetric}
x\in\Omega\Rightarrow\hat{x}\in\Omega,
\end{eqnarray}
and of class $C^{2+\alpha}$. If $\Omega$ is unbounded we require that $\Omega$ satisfies a uniform interior sphere condition and that the curvature of $\partial\Omega$ is bounded in the $C^\alpha$ sense.
\end{description}

 If $S\subset\R^d$ is a symmetric set, we define $S^+:=\{x\in S:x_1>0\}$.
We first consider the case of general $n\geq 1$ and prove the existence of a symmetric solution which is near $1$ in $\Omega^+$. Note that, in general,  $\partial(\Omega^+)\neq(\partial\Omega)^+$.
\begin{theorem}\label{main}
Assume that $W$ and $\Omega\subset\R^n$ satisfy h$_1$, h$_2$ and h$_3$.  Assume that $g:\partial\Omega\rightarrow\R$ is symmetric and bounded as a $C^{2,\alpha}(\partial\Omega;\R)$ function and satisfies
\[g(x)\geq 0,\;\text{ for }\; x\in(\partial\Omega)^+.\]
 Then, problem $(\ref{equation})$ has a symmetric classical solution $u\in C^2(\overline{\Omega};\R)$ such that \begin{equation}\label{u-properties}
\begin{array}{l}
 u(x)\geq 0,\;\text{ for }\; x\in\Omega^+ ,\\
 \displaystyle{
\vert u(x)-1\vert\leq K e^{-k d(x,\partial(\Omega^+))}, \quad x\in\Omega^+,}
\end{array}
\end{equation}
for some positive constants $k$, $K$ that depend only on $W, n$ and on the $C^1(\overline{\Omega};\R)$ norm of $u$.
\end{theorem}
\noindent (We assume that $g$ is extended to $\Omega$ as a symmetric $C^{2,\alpha}$ map).
A similar statement is valid in the case of Neumann boundary conditions.

We then restrict to the case $n=2$ and prove the following asymptotic result
\begin{theorem}\label{teo}
Assume $W$ as in Theorem $\ref{main}$ and assume that $\Omega\subset\R^2$ satisfies h$_3$  and is convex in $x_1$ i.e.
\begin{eqnarray}\label{x1-convex}
(x_1,x_2)\in\Omega\Rightarrow(t x_1,x_2)\in\Omega,\;\text{ for }\;\vert t\vert\leq 1.
\end{eqnarray}
Let $u$ be the solution of the Allen-Cahn equation $(\ref{equation})$ given by Theorem $\ref{main}.$ Then there exists a continuous decreasing map $R\rightarrow q(R),\;\;\lim_{R\rightarrow+\infty}q(R)=0$, such that
\begin{eqnarray}\label{near-ubar}
\vert u(x_1,x_2)-\bar{u}(x_1)\vert\leq q\Big(d(x,\partial\Omega)\Big),\;\text{ for }\;x\in\Omega,
\end{eqnarray}
where $\bar{u}:\R\rightarrow\R$ is the odd solution of
\begin{eqnarray}\label{odd-solution}
&& v^{\prime\prime}=W^\prime(v),\;\;s\in\R\\\nonumber
&& \lim_{s\rightarrow+\infty}v(s)=1.
\end{eqnarray}
The map $q$  depends only on $W, n$ and on the $C^1(\overline{\Omega};\R)$ norm of $u$.
\end{theorem}
A convergence result for odd solutions of (\ref{equation}) similar to (\ref{near-ubar}) valid in the case $\Omega\subset\R^n$ is a half space was obtained, among other things, in \cite{eh} (cfr. Theorem 1.1). The point in Theorem \ref{teo} is that, even though is restricted to $n=2$, applies to general domains that satisfy (\ref{x1-convex}).
Some of the ideas in the proof of Theorem \ref{teo} have been extended and utilized in \cite{af} where the restriction to $n=2$ is removed and $u$ is allowed to be a vector.

The proof of Theorems \ref{main} and Theorem \ref{teo} is variational and is based on a pointwise estimate for {\it local minimizers} of the Allen-Cahn energy
\begin{eqnarray}\label{ac-energy}
J_A(u):=\int_A(\frac{1}{2}\vert\nabla u\vert^2+W(u))dx,
\end{eqnarray}
defined for all bounded domain $A\subset\R^n$ and $u\in W^{1,2}(A;\R)$.
\begin{definition}\label{local-min}
Let $\Omega\subset\R^n$ be a domain. A map $u\in W_{\rm loc}^{1,2}(\Omega;\R)$ is a local minimizer of the Allen-Cahn energy if
\begin{eqnarray}\label{minimization}
J_A(u)=\min_{v\in W_0^{1,2}(A;\R)}J_A(u+v),
\end{eqnarray}
for every bounded Lipschitz domain $A\subset\Omega$.
\end{definition}
\noindent In the following we denote by $k, K$ and $C$ generic positive constants that can change from line to line.

The pointwise estimate alluded to above is stated in the following
\begin{theorem}\label{teo2}
Assume $W:\R\rightarrow\R$ is a $C^2$ function such that
\begin{description}
\item[(i)]$0=W(0)<W(v),\;\text{ for }\;v\in\R$,
\item[(ii)]$\liminf_{\vert v\vert\rightarrow+\infty}W(v)>0$,
\item[(iii)]$W^{\prime\prime}(0)>0$.
\end{description}
Let $\Omega\subset\R^n$ a domain and $u\in C^1(\Omega;\R)$ a local minimizer of the Allen-Cahn energy that satisfies
\begin{eqnarray}\label{bounds-udu}
\vert u\vert+\vert\nabla u\vert\leq M_0,\;\text{ for }\;x\in \Omega
\end{eqnarray}
and some $M_0>0$.

Then there is $q^*>0$ with the property that for each $q\in(0,q^*]$ there is $R(q)>0$ such that
\begin{eqnarray}\label{pointw-cond}
B_{x,R(q)}\subset \Omega\Rightarrow \vert u(x)\vert< q.
\end{eqnarray}
Moreover $R(q)$ can be chosen strictly decreasing and continuos in $(0,q^*]$. The inverse map $q(R)$ satisfies
\begin{eqnarray}\label{exp-bound}
q(R)\leq Ke^{-kR},\;\;R\in[R(q^*),+\infty),
\end{eqnarray}
for some positive constants $k, K$ that depend only on $W, n$ and the bound $M_0$.
\end{theorem}
The paper is organized as follows. In Sect. 2 we use Theorem \ref{teo2} to prove Theorem \ref{main}. In Sect. 3 we prove Theorem \ref{teo}. Finally in Sect. 4 we present a proof of Theorem \ref{teo2}. The proof is an adaptation to the scalar case of arguments developed in \cite{f} and  \cite{f1} for the vector Allen-Cahn equation.
\section{The proof of Theorem \ref{main}}
We first consider the case of $\Omega$ bounded. Then standard arguments from variational calculus yield the existence of a symmetric minimizer $u\in g+W_{0,S}^{1,2}(\Omega;\R)$ of $J_\Omega$; we denote by $W_{0,S}^{1,2}(\Omega;\R)$ the subspace of symmetric maps of $W_0^{1,2}(\Omega;\R)$. Let $g_m=\max_{(\partial\Omega)^+}g$. We can assume
\begin{eqnarray}\label{can-assume-u}
\vert u\vert\leq M^\prime:=\max\{M,g_m\},
\end{eqnarray}
and
\begin{eqnarray}\label{can-assume-u1}
 u\geq 0\;\text{ on }\;\Omega^+.
\end{eqnarray}
To prove (\ref{can-assume-u}) we let $v\in g+W_{0,S}^{1,2}(\Omega;\R)$ the symmetric function defined by
\begin{eqnarray}\label{def-vu}
v=\left\{\begin{array}{l}
u,\;\text{ on }\;\Omega^+\cap\{u\leq M^\prime\},\\
M^\prime,\;\text{ on }\;\Omega^+\cap\{u> M^\prime\},
\end{array}\right.
\end{eqnarray}
and observe that if $\Omega^+\cap\{u> M^\prime\}$ has positive measure, then h$_2$ implies
\begin{eqnarray}
J_\Omega(u)-J_\Omega(v)=\int_{\Omega^+\cap\{u> M^\prime\}}(\vert\nabla u\vert^2+2(W(u)-W(v))) dx>0,
\end{eqnarray}
in contradiction with the minimality of $u$. The proof of (\ref{can-assume-u1}) is similar.

From the bound (\ref{can-assume-u}), the smoothness assumption on $\partial\Omega$ in h$_3$ and elliptic regularity we obtain that $u$ is a classical solution of (\ref{equation}) and
\begin{eqnarray}\label{emme-bound}
\|u\|_{C^{2,\alpha}(\overline{\Omega};\R)}\leq M^{\prime\prime},
\end{eqnarray}
for some constant $M^{\prime\prime}>0$.
The restriction of $u$ to $\Omega^+$  trivially satisfies the definition of  minimizer of the Allen-Cahn energy in $\Omega^+$ with potential $\tilde{W}$ that, by (\ref{can-assume-u}) and (\ref{can-assume-u1}), can be identified with any smooth function that satisfies $\tilde{W}(s)=W(s)$, for $s\geq 0$ and  $\tilde{W}(s)>W(\vert s\vert)$, for $s<0$. From this and (\ref{emme-bound}) it follows that we can apply Theorem \ref{teo2} to $\hat{u}=u-1$ with potential $\tilde{W}(\cdot +1)$ and conclude that $u$ satisfies the exponential estimate
\begin{eqnarray}\label{e}
\vert u(x)-1\vert\leq K e^{-k d(x,\partial\Omega^+)},\;\text{ for }\;x\in\Omega^+,
\end{eqnarray}
with $k, K$ depending only on $W$ and $M^{\prime\prime}$. This concludes the proof for $\Omega$ bounded. If $\Omega$ is unbounded we consider a sequence of bounded domains $\Omega_j,\;j\in\nat ,$ such that $\Omega_j\subset\Omega_{j+1}$ and $\Omega=\cup_j\Omega_j$. From h$_3$ we can assume that the boundary of $\Omega_j$ is of class $C^{2,\alpha}$   and satisfies an interior sphere condition uniformly in $j\in\nat .$ Therefore the same reasoning developed for the case of bounded $\Omega$ yields
\begin{eqnarray}\label{emme-bound-j}
\|u_j\|_{C^{2,\alpha}(\overline{\Omega_j};\R)}\leq M^{\prime\prime},\;\text{ for }\;j\in\nat ,
\end{eqnarray}
and
\begin{eqnarray}\label{e-j}
\vert u_j(x)-1\vert\leq K e^{-k d(x,\partial\Omega_j^+)},\;\text{ for }\;x\in \Omega_j^+,\;j\in\nat .
\end{eqnarray}
The estimate (\ref{emme-bound-j}) implies that, passing to a subsequence if necessary, we can assume that $u_j$ converges locally in $C^2$ to a classical solution $u:\Omega\rightarrow\R$ of (\ref{equation}) and (\ref{e-j}) implies that $u$ satisfies the exponential estimate in Theorem \ref{main}. The proof of Theorem \ref{main} is complete.
\begin{remark}\label{remark}
Elliptic regularity implies that we can upgrade the exponential estimate in Theorem \ref{main} to
\begin{eqnarray}\label{exp-upgrade}
\vert u(x)-1\vert+\vert\nabla u(x)\vert\leq K e^{-k d(x,\partial\Omega^+)},\;\text{ for }\;x\in \Omega^+.
\end{eqnarray}
\end{remark}
In the proof of Theorem \ref{teo} below we make systematic use of the fact that the solution of (\ref{equation}) given by Theorem \ref{main} is a local minimizer in the sense of Definition \ref{local-min}. This is obvious when $\Omega$ is a bounded. If $\Omega$ is unbounded it follows from the fact that $u=\lim_{j\rightarrow+\infty}u_j$ is the limit of a sequence of minimizers $u_j:\Omega_j\rightarrow\R$, $\Omega_j$ bounded, that converges to $u$ uniformly in compacts \cite{s}.
\section{The proof of Theorem \ref{teo}}
We divide the proof in several lemmas.

For $l\in (0,+\infty]$ let
\begin{equation}\label{Bl}
{\mathcal B}_l:=\{\, v\in W^{1,2}_{\rm odd}((-l,l);\R )\, :\, v(\pm l)=0;\;
\Vert v\Vert_{1,l}\le M^{\prime\prime}\,\},
\end{equation}
where  $W^{1,2}_{\rm odd}((-l,l);\R )\subset  W^{1,2}((-l,l);\R )$ is the subset of the odd functions and $\Vert v\Vert_{1,l}$ is the usual $W^{1,2}$ norm of $v.$
Let ${\mathcal S}\subset  W^{1,2}_{\rm odd}((-l,l);\R )$ be defined by
\begin{equation}\label{ESSE}
{\mathcal S}:=\{\, \nu\in  W^{1,2}_{\rm odd}((-l,l);\R )\,:\, \Vert \nu\Vert_l=1\,\},
\end{equation}
where, for $l\in(0,+\infty]$, $\Vert v\Vert_l$ denotes the $L^2((-l,l);\R )$ norm of $v.$  In particular $\|v\|_\infty=\|v\|_{L^2(\R)}$.
For $v\in\mathcal{B}_l$ we write
\[v=q\nu,\;\text{ with }\;q=\|v\|_l\;\text {and }\;\nu\in\mathcal{S}.\]
For $w\in W^{1,2}((-l,l);\R)$ we set
\begin{equation}\label{smallE}
e_l(w)=\int_{-l}^l\Big(\frac{1}{2}\vert w_{x_1}\vert^2+W(w)\Big) dx_1.
\end{equation}
Define $E_l:{\mathcal B}_l\rightarrow \R$ by setting
\begin{equation}\label{Econl}
\begin{split}
& E_l(v)=e_l(\bar{u}+v)-e_l(\bar{u})\\&=\frac 1 2\int_{-l}^l (\vert \overline{u}_{x_1}+v_{x_1}\vert^2-\vert  \overline{u}_{x_1}\vert^2) dx_1 +\int_{-l}^l[W(\overline{u}+v)-W(\overline{u})] dx_1.
\end{split}
\end{equation}
\begin{lemma}\label{lemma1}
 There exist $q_0>0, c>0,$ independent of $l>1$, such that for  $v\in\mathcal{B}_l$, $v=q\nu$, we have
\begin{equation}\label{stimaa}
E_l(q\nu)\ge \frac 1 2 c^2q^2,\hspace{2 cm} 0<q\le q_0,\ \nu\in {\mathcal S},
\end{equation}
\begin{equation}\label{stimab}
E_l(q\nu)\ge \frac 1 2 c^2q_0^2,\hspace{2.6 cm} q_0\le q,\ \nu\in {\mathcal S}.
\end{equation}
Moreover it results
\begin{equation}\label{stimac}
D_{qq}E_l(q\nu)\ge c^2,\hspace{2.6 cm} 0<q\le q_0,\ \nu\in {\mathcal S}.
\end{equation}
\end{lemma}
\begin{proof}
From (\ref{odd-solution}) and $v(\pm l)=0$ it follows
\begin{equation}\label{E2}
\int_{-l}^l(\overline {u}_{x_1} v_{x_1}+W^\prime(\overline{u}) v)dx_1= \int_{-l}^l(-\overline{u}_{x_1x_1}+W^\prime(\overline{u}))v dx_1 =0.
\end{equation}
Therefore, for $v\in {\mathcal B}_l,$ we can rewrite $E_l(v)$ in the form

\begin{equation}\label{E3}
\begin{array}{l}
\displaystyle{
E_l(v)=\int_{-l}^l \Big (\frac 1 2 W^{\prime\prime}(\overline{u})v^2+\frac{v_{x_1}^2} 2\Big )dx_1
}\\
\hspace{1 cm}\displaystyle{
+\int_{-l}^l\Big [W(\overline {u}+v)-W(\overline{u})
-W^\prime(\overline{u})v-\frac 1 2 W^{\prime\prime}(\overline{u})v^2 \Big ] dx_1,}
\end{array}
\end{equation}
where we have also added and subtracted $\frac 1 2 W^{\prime\prime}(\overline{u})v^2$.

By differentiating (\ref{odd-solution}) we see that $\bar{u}_{x_1}$ is an eigenfunction corresponding to the eigenvalue $\lambda=0$ for the operator $L$ defined by

$$L:W^{1,2}( \R )\subset L^2(       \R)\rightarrow  L^2(       \R),$$
$$Lw:= -w_{x_1x_1}
+W^{\prime\prime}(\overline{u})w.$$

Since $\bar{u}$ is increasing and odd, $\bar{u}_{x_1}$ is positive and even. On the other hand the assumption $W^{\prime\prime}(\pm 1)>0$ implies, see e.g. Theorem A.2 of \cite{Henry} (pag. 140), that the essential spectrum of $L$ is contained in $[a,+\infty)$ for some $a>0$. Therefore, if we restrict to the subset of odd functions we can conclude that there exists a positive constant $c_1$ such that

\begin{equation}\label{E4}
\int_{-\infty}^{+\infty } \Big (\frac 1 2 W^{\prime\prime}(\overline{u})\phi^2+\frac{\phi_{x_1}^2} 2\Big )dx_1= \frac{1}{2}\int_{-\infty }^{+\infty} (L\phi ) \phi dx_1 \ge c_1^2\int_{-\infty}^{+\infty } \phi^2 dx_1,\;\text{ for all }\;\phi\in  W_{\rm odd}^{1,2}(\R ).
\end{equation}
In particular, given $v\in {\mathcal B}_l,$ we can apply (\ref{E4}) to the trivial extension $\tilde{v}$ of $v$ to obtain
  \begin{equation}\label{E4bis}
\int_{-l}^{+l } \Big (\frac 1 2 W^{\prime\prime}(\overline{u})v^2+\frac{v_{x_1}^2} 2\Big )dx_1\ge c_1^2\int_{-l}^{+l } v^2 dx_1,\;\text{ for all }\;v\in{\mathcal B}_l.
\end{equation}

Since $v\in{\mathcal B}_l$ implies $v(-l)=0$, we have $v^2(x_1)=2\int_{-l}^{x_1}v(s)v_{x_1}(s) ds$ and therefore
\begin{equation}\label{NG}
 \Vert v\Vert_{L^\infty(-l,l)}\leq \sqrt{2}\Vert v\Vert_l^{1/2}\Vert v\Vert_{1,l}^{1/2}\leq C\Vert v\Vert_l^{1/2},\;\text{ for }\;v\in{\mathcal B}_l,
\end{equation}
with $C=\sqrt{2 M^{\prime\prime}}$. Fix $q_0>0$ and let $\overline{W}^{\prime\prime\prime}=\max_{\vert s\vert\leq 1+C q_0^{1/2}}\vert W^{\prime\prime\prime}(s)\vert$. Then, for some map $x_1\rightarrow\theta(x_1)\in(0,1)$, we have

\begin{eqnarray}\label{E5}\\\nonumber
&&\left\vert\int_{-l}^l(W(\overline {u}+v)-W(\overline{u})
-W^\prime(\overline{u})v-\frac 1 2 W^{\prime\prime}(\overline{u})v^2) d x_1\right \vert
=\left \vert\int_{-l}^l W^{\prime\prime\prime}(\overline{u}+\theta v)\frac {v^3} 6 d x_1\right\vert\\\nonumber
&&\hskip6cm\leq\frac{1}{6}C\overline{W}^{\prime\prime\prime}q_0^{1/2}\int_{-l}^lv^2 d x_1,\;\text{ for }\;\Vert v\Vert_l\leq q_0.
\end{eqnarray}
From (\ref{E3}), (\ref{E4bis}) and (\ref{E5}), if we choose $q_0>0$ so small that $C\overline{W}^{\prime\prime\prime}q_0^{1/2}\leq 3c_1^2$, it follows

$$E_l(q\nu)=E_l(v)\ge \frac 1 2 c_1^2\int_{-l}^l v^2 dx_1,\;\text{ for }\;v\in{\mathcal B}_l,\;0<q=\Vert v\Vert_l\leq q_0 ,$$
that is
(\ref{stimaa}).

To show (\ref{stimab}) let us consider the minimization problem
\begin{equation}\label{Pstar}
\min_{
\begin{array}{l}
v\in {\mathcal B}_l
\\
\Vert v\Vert_l\ge q_0
\end{array}
} E_l(v)\,.
\end{equation}
It is easy to construct a smooth odd map $w\in{\mathcal B}_l$ that satisfies the constraint $\Vert w\Vert_l\ge q_0$. Therefore there exists a minimizing sequence $\{v_j\}\subset{\mathcal B}_l$ that satisfies $\Vert v_j\Vert_l\ge q_0,\;j\in\nat ,$ and
\begin{eqnarray}\label{bound-constraint}
E_l(v_j)\leq E_l(w),\;j\in\nat .
\end{eqnarray}
From (\ref{bound-constraint}) and standard arguments from variational calculus it follows that there is $\overline{v}_l\in{\mathcal B}_l$ and a subsequence $\{v_{j_h}\}$ such that
\begin{eqnarray}\label{liminf-seq}
&&\liminf_{h\rightarrow+\infty}E_l(v_{j_h})\geq E_l(\overline{v}_l),\\\nonumber
&&\lim_{h\rightarrow+\infty}\Vert v_{j_h}- \overline{v}_l\Vert_l=0.
\end{eqnarray}
It follows $\Vert \overline{v}_l\Vert_l\ge q_0$ and $\overline{v}_l$ is a minimizer of (\ref{Pstar}).
Since $E_l(0)=0$ and $v=0$ is the unique minimizer of $E_l$ on ${\mathcal B}_l,$ this implies $E_l(\overline {v}_l)=\alpha_l>0,$ and therefore
\begin{equation}\label{E6}
E_l(q\nu )\ge \alpha_l, \quad\mbox{for}\ q\ge q_0.
\end{equation}
Note that $\alpha_l$ is non increasing  with $l.$
Indeed, if $l_1<l_2$ and $v\in {\mathcal B}_{l_1},$ then the trivial extension $\tilde{v}$ of $v$ to $[-l_2,l_2]$ satisfies $E_{l_2}(\tilde{v})=E_{l_1}(v)$ and belongs to ${\mathcal B}_{l_2}.$
Therefore, there exists $\lim_{l\rightarrow +\infty} \alpha_l.$
We claim that
\begin{equation}\label{limitpos}
\lim_{l\rightarrow +\infty} \alpha_l=\alpha >0.
\end{equation}

Let $\{l_k\}_k$ be a sequence of positive numbers such that
$l_k\rightarrow +\infty$ for $k\rightarrow +\infty$.
Let $\overline{v}_k$ be a minimizer of problem (\ref{Pstar})
for $l=l_k$ and let $\tilde{\overline{v}}_k:\R\rightarrow\R$ be the trivial extension of
$\overline{v}_k,$ we may assume that the sequence $\{\tilde{\overline{v}}_k\}_k$ converges in $L^2(\R)$ and weakly
in $W^{1,2}(\R )$ to a map $\overline{v}$ which satisfies, by lower semicontinuity of $E_\infty$,
$\Vert\overline{v}\Vert_\infty\geq q_0$ and $\alpha\geq E_\infty(\overline{v})$. Since $v=0$ is the unique minimizer of $E_\infty ,$ this implies $\alpha\geq E_\infty(\overline{v})>0$ and proves
 (\ref{limitpos}).
From (\ref{E6}) we then deduce
$$
E_l (q\nu)\ge \alpha =\frac 1 2 c_2^2 q_0^2,$$
for a suitable constant $c_2$ independent of $l.$
Then both (\ref{stimaa}) and (\ref{stimab}) hold with $c:=\min \{c_1, c_2\}.$

To prove (\ref{stimac}) we note that  setting $v=q\nu$ in (\ref{E3}) yields
\begin{equation}\label{E31}
\begin{array}{l}
\displaystyle{
E_l(q\nu)=q^2\int_{-l}^l \Big (\frac 1 2 W^{\prime\prime}(\overline{u})\nu^2+\frac{\nu_{x_1}^2} 2\Big )dx_1
}\\
\hspace{1 cm}\displaystyle{
+\int_{-l}^l\Big [W(\overline {u}+q\nu)-W(\overline{u})
-q W^\prime(\overline{u})\nu-\frac 1 2 q^2W^{\prime\prime}(\overline{u})\nu^2 \Big ] dx_1,}
\end{array}
\end{equation}
which via (\ref{E4bis}) implies
\begin{eqnarray*}
D_{qq}E_l(q\nu)|_{q=0}=\int_{-l}^l \Big ( W^{\prime\prime}(\overline{u})\nu^2+\nu_{x_1}^2\Big )dx_1\geq 2c_1^2.
\end{eqnarray*}
This concludes the proof.
\end{proof}
For  $r>0, l>0$ and $\eta\in\R ,$
we denote by $C_l^r(\eta)$ the set
\begin{equation}\label{Clr}
C_l^r(\eta):=
\{(x_1, x_2)\in\R^2\,:\, \vert x_1\vert< l,\, \vert x_2-\eta\vert< r\}.
\end{equation}
In the following, whenever possible, we assume that by a translation we can reduce to the case $\eta=0$ and write simply $C_l^r$ instead of $C_l^r(0)$.

The introduction of the map $E_l$ allows to represent the energy $J_{C_l^r}(v)$ of an odd map $v:C_l^r\rightarrow\R$  that satisfies $v(x)=\bar{u}(x_1),\;\text{ for }\;x_1=\pm l,\;\vert x_2\vert<r$ in a particular form that we now derive. We have
\begin{eqnarray}\label{first-energy}
J_{C_l^r}(v)=\frac{1}{2}\int_{-r}^r\int_{-l}^l\vert v_{x_2}\vert^2 dx_1 dx_2+\int_{-r}^r E_l(v-\bar{u}) dx_2+\int_{-r}^re_l(\bar{u}) dx_2.
\end{eqnarray}
If we set
\[q^v(x_2):=\|v(\cdot,x_2)-\bar{u}(\cdot)\|_l>0\]
then  $v-\bar{u}=q^v\nu^v$ with $\nu^v=\frac{v-\bar{u}}{\|v-\bar{u}\|_l}$. We observe that $v_{x_2}=q_{x_2}^v\nu^v+q^v\nu_{x_2}^v$ implies
\begin{eqnarray}\label{kinetic-decomposition}
\int_{-l}^l\vert v_{x_2}\vert^2 dx_1=\vert q_{x_2}^v\vert^2+(q^v)^2\int_{-l}^l\vert\nu_{x_2}^v\vert^2 dx_1,
\end{eqnarray}
where we have also used $\int_{-l}^l\nu_{x_2}^v\nu^v dx_1=0$. It follows
\begin{eqnarray}\label{last-energy0}
\hspace{0.8cm}
J_{C_l^r}(v)=\frac{1}{2}\int_{-r}^r(\vert q_{x_2}^v\vert^2+(q^v)^2\|\nu_{x_2}^v\|_l^2) dx_2+\int_{-r}^rE_l(q^v\nu^v) dx_2+\int_{-r}^re_l(\bar{u}) dx_2.
\end{eqnarray}
Assume now that $w:C_l^r\rightarrow\R$  is of the form \[w(x_1,x_2)=\bar{u}(x_1)+q^w(x_2)\nu^v(x_1,x_2) ;\]
 then we have
\begin{eqnarray}\label{last-energy}
\hspace{0.5cm}
J_{C_l^r}(w)=\frac{1}{2}\int_{-r}^r(\vert q_{x_2}^w\vert^2+(q^w)^2\|\nu_{x_2}^v\|_l^2) dx_2+\int_{-r}^rE_l(q^w\nu^v) dx_2+\int_{-r}^re_l(\bar{u}) dx_2.
\end{eqnarray}
For later reference we state
\begin{lemma}\label{linfty}
Let $f:[-l,l]\rightarrow\R$ be a Lipschitz continuous function satisfying
\begin{eqnarray}\label{assumption}
\vert f(s)\vert+\vert f^\prime(s)\vert\leq K e^{-k\vert s\vert},\;\text{ for }\; s\in(-l,l).
\end{eqnarray}
Then,  there is a constant $C_2>0$ independent of $l\geq1$ such that
\begin{eqnarray}\label{linfty1}
\|f\|_{L^\infty}\leq C_2\|f\|_l^\frac{2}{3}
\end{eqnarray}
\end{lemma}
\begin{proof}
From (\ref{assumption}) there is $\bar{s}\in[-l,l]$ such that $\vert f(s)\vert\leq m:=\vert f(\bar{s})\vert,\;s\in[-l,l]$. From this and $\vert f^\prime(s)\vert\leq K$ it follows
\[\vert f(s)\vert\geq m-K\vert s-\bar{s}\vert,\;\text{ for }\;s\in[-l,l]\cap[\bar{s}-m/K,\bar{s}+m/K]\]
and a simple computation gives (\ref{linfty1}).
\end{proof}

\begin{lemma}\label{elle-zero}
There exist $J_0>0,\ C>0,$ $\ k>0$ and a map $(0,\infty)\ni r\rightarrow l_r>0$ such that, given $r>0$, if $l\ge l_r$ and
\begin{equation}\label{Rlarge}
C_l^r\subset\Omega,\;\;d(C_l^r, \partial \Omega)>l ,
\end{equation} then there is
a Lipschitz continuous function $v$ with the following properties:
\begin{description}
\item[(i)] $v(x)=\overline{u} (x_1),\quad$ for $x\in
\partial C_{l+\delta/2}^{r+\delta/2};$
\item[(ii)] $v(x)=u(x),\quad$ for $x\in \overline{C}_{l,
r}$ and $x\in \Omega\setminus C_{l+\delta}^{r+\delta};$
\item[(iii)] $\|v(\cdot,x_2)-u(\cdot,x_2)\|_{l+\delta/2}\leq C e^{-k l},\;\text{ for }\;x_2\in[-r,r],$
\item[(iv)] $J_{\overline{C_{l+\delta}^{r+\delta}}\setminus C_l^r}(v)-J_{\overline{C_{l+\delta}^{r+\delta}}\setminus C_l^r}(u)\le J_0,$
\end{description}
where $\delta>0$ is a fixed constant.
 \end{lemma}
\begin{proof}
We set
\begin{eqnarray}\label{cr-cr}
v=u\;\text{ for }\;x\in \overline{C_l^r}\cup(\Omega\setminus C_{l+\delta}^{r+\delta}).
\end{eqnarray}
To define $v$ in $C_{l+\delta}^{r+\delta}\setminus\overline{C_l^r}$
let $S_1\subset\R^2$ be the sector $S_1=\{x:x_1\geq l-r,\;\vert x_2\vert< x_1-l+r\}$ and let $(\rho,\theta)$ polar coordinates in $S_1$ with origin in the vertex $(l-r,0)$ of $S_1$ and polar axis parallel to $x_1$. We let $x(\rho,\theta)$ denote the point of $S_1$ with polar coordinates $(\rho,\theta)$. We define $v$ in the trapezoid $T_1:=(C_{l+\delta}^{r+\delta}\setminus\overline{C_l^r})\cap S_1$ by setting
\begin{equation}\label{defv1}
\begin{array}{l}
v(x(\rho,\theta)):=
\Big (
1-\Big\vert 1-2\frac {\rho-\rho_1(\theta)} {\rho_2(\theta)-\rho_1(\theta)}\Big\vert \Big )\overline{u} (l+\delta/2)
+\Big\vert  1-2\frac {\rho-\rho_1(\theta)} {\rho_2(\theta)-\rho_1(\theta)}\Big\vert u(x(\rho,\theta)),\\\\
\medskip
\quad \hspace{6.5 cm}\mbox{\rm for }\
\rho\in (\rho_1(\theta),\rho_2(\theta)),\: \vert\theta\vert\leq\frac{\pi}{4},
\end{array}
\end{equation}
where $\rho_1(\theta)$, and $\rho_2(\theta)$ are defined by the conditions $x_1(\rho_1(\theta),\theta)=l$ and $x_1(\rho_2(\theta),\theta)=l+\delta$.

In the trapezoid $T_2:=(C_{l+\delta}^{r+\delta}\setminus\overline{C_l^r})\cap S_2$, $S_2=\{x:x_2\geq r-l,\;\vert x_1\vert< x_2-r+l\}$ we define

\begin{equation}\label{defv2}
\begin{array}{l}
v(x(\varrho,\phi)):=
\Big (
1-\Big\vert 1-2\frac {\varrho-\varrho_1(\phi)} {\varrho_2(\phi)-\varrho_1(\phi)}\Big\vert \Big )\overline{u} (x_1(\frac{\varrho_1(\phi)+\varrho_2(\phi)}{2},\phi))\\\\
\hspace{4cm}
+\Big\vert  1-2\frac {\varrho-\varrho_1(\phi)} {\varrho_2(\phi)-\varrho_1(\phi)}\Big\vert u(x(\varrho,\phi)),
\quad \mbox{\rm for }\
\varrho\in (\varrho_1(\phi),\varrho_2(\phi)),\: \vert\phi\vert\leq\frac{\pi}{4},
\end{array}
\end{equation}
where $(\varrho,\phi)$ are polar coordinates in $S_2$ and $\varrho_1(\phi)$, and $\varrho_2(\phi)$ are defined by the conditions $x_2(\varrho_1(\phi),\phi)=r$ and $x_2(\varrho_2(\phi),\phi)=r+\delta$. In the remaining two trapezoids we define $v$ in a similar way.

The maps defined by (\ref{cr-cr}), (\ref{defv1}) and (\ref{defv2}) are Lipschitz continuous in the closure of their domains of definition and join continuously on the boundary of $C_{l+\delta}^{r+\delta}\setminus\overline{C_l^r}$ and along the line $\theta=\pi/4$. Indeed (\ref{defv1}) and (\ref{defv2}) yield
\begin{eqnarray*}
\left.\begin{array}{l}
\rho=\rho_i(\theta)\\
\varrho=\varrho_i(\phi)
\end{array}\right.
\Rightarrow\left.\begin{array}{l}
v(x(\rho_i(\theta),\theta))=u(x(\rho_i(\theta),\theta)),\\
v(x(\varrho_i(\phi),\phi))=u(x(\varrho_i(\phi),\phi))
\end{array}\right.\;i=1,2.
\end{eqnarray*}
and
\begin{eqnarray*}
&& x(s+\rho_1(\pi/4),\pi/4)=x(s+\varrho_1(\pi/4),\pi/4),\\\Rightarrow\\
&&v(x(s+\rho_1(\pi/4),\pi/4))=v(x(s+\varrho_1(\pi/4),\pi/4)),\;\;s\in[0,\sqrt{2}\delta].
\end{eqnarray*}
Therefore we conclude that, as defined, the map $v$ is uniformly Lipschitz continuous on $\Omega$.
The fact that $v$ satisfies (i) follows from (\ref{defv1}) and (\ref{defv2}) that imply
\begin{eqnarray*}
\left.\begin{array}{l}
\rho=(\rho_1(\theta)+\rho_2(\theta))/2\\
\varrho=(\varrho_1(\phi)+\varrho_2(\phi))/2
\end{array}\right.
\Rightarrow\left.\begin{array}{l}
v(x((\rho_1(\theta)+\rho_2(\theta))/2,\theta))=\overline{u}(l+\delta/2),\\
v(x((\varrho_1(\phi)+\varrho_2(\phi))/2,\phi))=\overline{u}(x_1((\varrho_1(\phi)+\varrho_2(\phi))/2)).
\end{array}\right.
\end{eqnarray*}

To prove (iii) and (iv) we use the estimate
\begin{eqnarray}\label{exp-ubar}
\vert\overline{u}(s)-1\vert+\vert\overline{u}^\prime(s)\vert\leq K e^{-k s},\;\text{ for }\;s\geq 0,
\end{eqnarray}
and the estimate for the solution $u$ established in (\ref{exp-upgrade}). Set $\lambda:=\Big\vert 1-2\frac {\rho-\rho_1(\theta)} {\rho_2(\theta)-\rho_1(\theta)}\Big\vert\in[0,1]$, then (\ref{defv1}) implies
\begin{eqnarray}\label{with-lambda}
v-1=(1-\lambda)(\overline{u}(l+\delta/2)-1)+\lambda (u(x(\rho,\theta))-1).
\end{eqnarray}
This, $x_1(\rho,\theta)>l\;\text{ on }\;T_1$, (\ref{exp-ubar}) and (\ref{exp-upgrade}) imply $\vert v-1\vert\leq K e^{-kl}\;\text{ on }\;T_1$ and therefore
(iii) follows. Moreover,
since $W(s)=O((s-1)^2)$ for $s-1$ small, it results
\begin{eqnarray}\label{wv-wu}
\int_{T_1}(W(v)-W(u)) dx\leq \int_{T_1}W(v) dx\leq C r \delta e^{-\gamma l},
\end{eqnarray}
where $\gamma, C$ denote a generic positive constants independent of $r$ and $l$. Differentiating (\ref{defv1}) in $x$ yields
\begin{eqnarray}\label{dv-du}
  \nabla v = (1-\lambda)\overline{u}^\prime(l+\delta/2)e_1+\lambda\nabla u(x(\rho,\theta))
  -(\overline{u}(l+\delta/2)-u(x(\rho,\theta)))\nabla\lambda,
\end{eqnarray}
where $e_i,\;i=1,2$ is the standard basis of $\R^2$. Since $\nabla\lambda$ is bounded on $T_1$ with a bound independent of $r$ and $l$, using again (\ref{exp-ubar}) and (\ref{exp-upgrade}) we see that (\ref{dv-du}) implies
\begin{equation}\label{dv-dv1}
   \int_{T_1}\frac{1}{2}(\vert\nabla v\vert^2-\vert\nabla u\vert^2) dx\leq\int_{T_1}\frac{1}{2}\vert\nabla v\vert^2 dx
   \leq C r \delta e^{-\gamma l}.
\end{equation}
To estimate $J_{T_2}(v)-J_{T_2}(u)$ we proceed in a similar way. We set $\lambda=\Big\vert 1-2\frac {\varrho-\varrho_1(\phi)} {\varrho_2(\phi)-\varrho_1(\phi)}\Big\vert$ and write equations analogous to (\ref{with-lambda}) and (\ref{dv-du}). From these equations, using as before the estimates (\ref{exp-ubar}) and (\ref{exp-upgrade}), and observing that
\begin{equation}\label{x1-large}
    \varrho\in(\varrho_1(\phi), \varrho_2(\phi))\;\Rightarrow\;\vert x_1(\varrho,\phi)\vert\geq\vert x_1(\varrho_1(\phi),\phi)\vert=l\vert\tan{\phi}\vert,
\end{equation}
it follows that there is a constant $C_0$ independent of $r$ and $l$ such that $J_{T_2}(v)-J_{T_2}(u)\leq J_{T_2}(v)\leq C_0$. This,  (\ref{wv-wu}) and (\ref{dv-dv1}) imply

\begin{equation}\label{ver1}
J_{\overline{C_{l+\delta}^{r+\delta}}\setminus C_l^r}(v)-J_{\overline{C_{l+\delta}^{r+\delta}}\setminus C_l^r}(u)\leq J_{\overline{C_{l+\delta}^{r+\delta}}\setminus C_l^r}(v)\leq 2(C_0+C r \delta e^{-\gamma l})
\end{equation}
and (iv) follows with $J_0=4 C_0$ and $l_r=-\frac{1}{\gamma}\ln{\frac{C r \delta}{C_0}}$.
\end{proof}

Arguments analogous to the ones in the proof of Lemma \ref{elle-zero} prove
\begin{lemma}\label{elle-zero1}
Assume that $C_l^r$ satisfies $(\ref{Rlarge}).$ Then there is
a Lipschitz continuous function $v$ with the following properties:
\begin{description}
\item[(i)] $v(x)=\overline{u} (x_1),\quad$ for $x\in\{-l-\delta/2,l+\delta/2\}\times[-r-\delta/2,r+\delta/2],$
\item[(ii)] $v(x)=u(x),\quad$ for $x\in\Omega\setminus((-l-\delta,-l)\cup(l,l+\delta))\times[-r-\delta,r+\delta], $
\item[(iii)] $\|v(\cdot,x_2)-u(\cdot,x_2)\|_{l+\delta/2}\leq C e^{-\gamma l},\;\text{ for }\;x_2\in[-r-\delta/2,r+\delta/2],$
\item[(iv)] $J(v)-J(u)\le C r e^{-\gamma l},$
\end{description}
for some constants $C, \gamma>0$.
 \end{lemma}

\begin{lemma}\label{a-set} Let $q_0$ and $c$ be as in Lemma $\ref{lemma1}.$
Given $\overline{q}< q_0,$  fix $r>0$ such that
\begin{equation}\label{venti1}
\frac 1 2 c^2 \overline{q}^2 r>8 J_0,
\end{equation}
where $J_0$ is the constant in (iv) in Lemma $\ref{elle-zero}.$ There is $l(\bar{q})>0$ such that, provided $(\ref{Rlarge})$ is satisfied with $l\geq\max\{l_r,l(\bar{q})\}$,
then there exist $a_-\in(-r,-r/2)$ and $a_+\in(r/2,r)$ such that
\begin{eqnarray}\label{the-lines}
\|u(\cdot,a_\pm)-\bar{u}\|_{l+\delta/2}<\bar{q}.
\end{eqnarray}
\end{lemma}
\begin{proof}
  Let $v$ the map constructed in Lemma \ref{elle-zero}.
 For each $\eta\in[-r,r/2]$ let $\mathcal{A}_\eta\subset\R$ be the set
\begin{equation}\label{calA}
{\mathcal A}_\eta :=\Big \{ x_2\in(\eta,\eta+r/2)\,:\,q^v(x_2)=\|v(\cdot,x_2)-\overline{u}\|_{l+\delta/2}
\geq\frac{\bar{q}}{2}\Big \}.
\end{equation}
Then, we have
\begin{equation}\label{venti2}
J_{C_{l+\delta /2}^{r+\delta /2}}(v)-J_{C_{l+\delta /2}^{r+\delta /2}}(\hat v )\ge \left\vert {\mathcal A}_\eta\right\vert \frac 1 2 c^2\frac{\overline{q}^2}{4},\;\text{ for }\;\eta\in[-r,r/2]
\,,
\end{equation}
where  $\hat v$ be the function that coincides with
$v$  outside $C_{l+\delta /2}^{r+\delta /2}$
and with $\bar{u}$ inside $C_{l+\delta /2}^{r+\delta /2}.$
Note that, since $v$ coincides with $\bar{u}$ on the boundary of $C_{l+\delta /2}^{r+\delta /2}$, $\hat{v}$ is a Lipschitz map.
To prove (\ref{venti2}), we observe that from the definition of $\hat{v}$ and of $E_l$ in (\ref{Econl})
we have (with $w=v-\bar{u}$)
$$
\begin{array}{l}
\displaystyle{J_{C_{l+\delta /2}^{r+\delta /2}}(v)-J_{C_{l+\delta /2}^{r+\delta /2}}(\hat v)=
\frac 1 2\int_{C_{l+\delta /2}^{r+\delta /2}}\vert w_{x_2}\vert^2 dx_1 dx_2+\int_{-r-\delta /2}^{r+\delta /2}E_{l+\delta /2}(w) dx_2
}\\
\hspace{4.5 cm}\displaystyle{
\geq\int_{-r}^r E_{l+\delta /2}(w) dx_2\geq\frac{1}{2}\vert{\mathcal A}_\eta\vert c^2\frac{\overline{q}^2}{4},\;\text{ for }\;\eta\in[-r,r/2]
}
\end{array}
$$
where we have also used (\ref{calA}) and  (\ref{stimaa}), (\ref{stimab}) in Lemma \ref{lemma1}.
Then, from Lemma \ref{elle-zero} and (\ref{venti2}), it follows
\begin{equation}\label{venti3}
0\ge J_{C_{l+\delta /2}^{r+\delta /2}}(v)-J_{C_{l+\delta /2}^{r+\delta /2}}(\hat v)\ge \left\vert {\mathcal A}_\eta\right\vert \frac 1 2 c^2\frac{\overline{q}^2}{4}-J_0> (\left\vert {\mathcal A}_\eta\right\vert-\frac{r}{2} )\frac 1 2 c^2 \frac{\overline{q}^2}{4},\;\text{ for }\;\eta\in[-r,r/2]
\end{equation}
and therefore
\begin{equation}\label{venti5}
\left\vert {\mathcal A}_\eta\right\vert <\frac{r}{2},\;\text{ for }\;\eta\in[-r,r/2]\,.
\end{equation}
This inequality and the definition (\ref{calA}) of ${\mathcal A}_\eta$ imply the existence of $a_-\in(-r,-r/2)\setminus{\mathcal A}_0$ and $a_+\in(r/2,r)\setminus{\mathcal A}_{3r/2}$ such that
\begin{eqnarray}
\|v(\cdot,a_\pm)-\bar{u}\|_{l+\delta /2}<\frac{\bar{q}}{2}.
\end{eqnarray}
This and (iii) in Lemma \ref{elle-zero} imply (\ref{the-lines}) provided $l\geq l(\bar{q}):=\frac{1}{k}\ln{\frac{2 C}{\bar{q}}}$.
\end{proof}

\begin{lemma}\label{de-exists}
Given $\epsilon>0$ there is $l_\epsilon$ such that
\begin{eqnarray*}
x\in\Omega,\;\;d(x,\partial\Omega)\geq l_\epsilon\;\Rightarrow\;\vert u(x)-\bar{u}(x_1)\vert\leq\epsilon.
\end{eqnarray*}
\end{lemma}
\begin{proof}
Set $d_\epsilon:=\frac{1}{k}\ln{\frac{2K}{\epsilon}}$  and assume that $d(x,\partial\Omega^+)\geq d_\epsilon$. Then (\ref{u-properties})$_2$ and (\ref{exp-ubar}) imply
\[\vert u(x)-\bar{u}(x_1)\vert\leq\vert u(x)-1\vert+\vert1-\bar{u}(x_1)\vert\leq\epsilon.\]
This and the oddness of $u$ imply that it suffices to consider the points $x\in\Omega^+$  which have  $d(x,\partial\Omega)\geq d_\epsilon$ and $x_1\in[0,d_\epsilon]$. Assume $\tilde{x}\in\Omega^+$ is a point with these properties that satisfies $\vert u(\tilde{x})-\bar{u}(\tilde{x}_1)\vert>\epsilon$. Then from (\ref{emme-bound}) and (\ref{exp-ubar}) it follows
\begin{equation}\label{extra}
\vert u(\tilde x)-\bar{u}(\tilde x_1)\vert -
\vert u(x)-\bar{u}(x_1)\vert
\le 2\mu (\vert x_1-\tilde x_1\vert +\vert x_2-\tilde x_2\vert ),
\end{equation}
where $\mu :=\max\{M^{\prime\prime},K\}$.
Then,
\begin{eqnarray}\label{l2-big}
\vert u(x)-\bar{u}(x_1)\vert >\frac {\epsilon } {2},\quad \text{ for }\ \vert x_1-
\tilde{x}_1\vert<\epsilon/{8 \mu},\ \vert x_2-\tilde{x}_2\vert<\epsilon/{8 \mu}.
\end{eqnarray}
From this inequality it follows
\begin{eqnarray}\label{l2-big1}
\|u(\cdot,x_2)-\bar{u}\|_{l+\delta /2}\geq \frac{1}{4\sqrt{\mu}}\epsilon^\frac{3}{2},\;\text{ for }\;\vert x_2-\tilde{x}_2\vert<\epsilon/{8\mu}
\end{eqnarray}
and thus, recalling Lemma \ref{elle-zero1} (iii)
\begin{eqnarray}\label{l2-big2}
\|v(\cdot,x_2)-\bar{u}\|_{l+\delta /2}\geq \frac{1}{8\sqrt{\mu}}\epsilon^\frac{3}{2},\;\text{ for }\;\vert x_2-\tilde{x}_2\vert<\epsilon/{8\mu}.
\end{eqnarray}
Set $q^*:=\frac{1}{4\sqrt{ \mu}}\epsilon^\frac{3}{2}$ and $\bar{q}=q^*/N$ where $N>0$ is a fixed number to be chosen later. In the remaining part of the proof we consider a certain number of lower bounds for $l$ and we always assume that (\ref{Rlarge}) is satisfied for $l> l_M$ where $l_M$ represents the maximum of the values  $l_r, l(\bar{q}),\dots$ introduced up the the point considered in the proof.

From Lemma \ref{a-set}, if $N$ is such that $\bar{q}<q_0,$  there is $r$ such that, for $l$ sufficiently large, there exist $a_-\in(\tilde{x}_2-r,\tilde{x}_2-r/2)$ and $a_+\in(\tilde{x}_2+r/2,\tilde{x}_2+r)$ with the property
\begin{eqnarray}\label{small-at-a}
\|u(\cdot,a_\pm)-\bar{u}\|_{l+\delta/2}<\bar{q}.
\end{eqnarray}
Moreover, from Lemma \ref{elle-zero1}, for $l$ sufficiently large, the map $v$ defined in the lemma satisfies $\|u(\cdot,a_\pm)-v(\cdot,a_\pm)\|_{l+\delta/2}<\bar{q}$ and therefore we have
\begin{eqnarray}\label{small-at-a1}
\|v(\cdot,a_\pm)-\bar{u}\|_{l+\delta/2}<2\bar{q}.
\end{eqnarray}
Let $Q:=(-l-\delta/2,l+\delta/2)\times(a_-,a_+)$ and let $w$ the map defined by
\begin{eqnarray}\label{w-def}
w=\left\{\begin{array}{l}
v,\;\text{ on }\;\Omega\setminus Q,\\
v,\;\text{ on }\;(-l-\delta/2,l+\delta/2)\times\{x_2\},\;x_2\in(a_-,a_+)\\
\hskip4cm\text{ if }\;q^v(x_2)\leq 2\bar{q},\\
\bar{u}+2\bar{q}\nu^v,\;\text{ on }\;(-l-\delta/2,l+\delta/2)\times\{x_2\},\;x_2\in(a_-,a_+)\\
\hskip4cm\text{ if }\;q^v(x_2)> 2\bar{q}.
\end{array}\right.
\end{eqnarray}
This definition implies in particular
\begin{eqnarray*}
\|w(\cdot,\tilde{x}_2)-\bar{u}\|_{l+\delta/2}\leq 2\bar{q}=\frac{2}{N}\frac{1}{4\sqrt{\mu}}\epsilon^\frac{3}{2}.
\end{eqnarray*}
Then Lemma \ref{linfty}, provided $N$ is chosen sufficiently large, implies
\begin{eqnarray}\label{small-at-x2}
\vert w(\tilde{x})-\bar{u}(\tilde{x}_1)\vert\leq C_2(\frac{2}{N}\frac{1}{4\sqrt
{\mu}})^\frac{2}{3}\epsilon<\epsilon.
\end{eqnarray}
On the other hand (\ref{w-def}),
 (\ref{last-energy0}) and (\ref{last-energy}) imply
\begin{eqnarray}\label{e-comparison}\\\nonumber
J_Q(v)-J_Q(w)
&=&\int_{\{q^v>2\bar{q}\}}[ \frac{1}{2}(\vert q_{x_2}^v\vert^2+((q^v)^2-4\bar{q}^2)\|\nu^v\|_{l+\delta/2}^2)
\\\nonumber &\ &\hspace{3cm}+E_{l+\delta/2}(q^v\nu^v)-
E_{l+\delta/2}(2\bar{q}\nu^v)] d x_2
\\\nonumber &\geq &\int_{\{q^v>2\bar{q}\}} [E_{l+\delta/2}(q^v\nu^v)-E_{l+\delta/2}(2\bar{q}\nu^v)] dx_2
\\\nonumber &\geq &\int_{\{q^v>q^*\}}[E_{l+\delta/2}(q^*\nu^v)-E_{l+\delta/2}(2\bar{q}\nu^v)] d x_2
.
\end{eqnarray}
 From (\ref{stimac}), for $q\leq q_0$,  we have $D_qE_l(q\nu)\geq c^2 q$ and therefore, recalling also that $\bar{q}=q^*/N$, we have
\begin{eqnarray}\label{e-e}
E_{l+\delta/2}(q^*\nu^v)-E_{l+\delta/2}(2\bar{q}\nu^v)
\geq\frac{1}{2}c^2(q^*)^2(1-\frac{4}{N^2})
\end{eqnarray}
which via (\ref{l2-big2}) yields
\[\int_{\{q^v>q^*\}} [E_{l+\delta/2}(q^*\nu^v)-E_{l+\delta/2}(2\bar{q}\nu^v)] dx_2
\geq\frac{1}{2}c^2(q^*)^2(1-\frac{4}{N^2})\frac{\epsilon}{4\mu} .\]
Then,  from (\ref{e-comparison}) and $q^*=\frac{1}{4\sqrt{\mu}}\epsilon^\frac{3}{2}$ we obtain
\begin{eqnarray*}
J_Q(v)-J_Q(w)\geq\frac{c^2}{128\mu^2}(1-\frac{4}{N^2})\epsilon^4.
\end{eqnarray*}
From this and Lemma \ref{elle-zero1} (iv) it follows
\begin{eqnarray}\label{e-e1}
J_Q(u)-J_Q(w)=J_Q(u)-J_Q(v)+J_Q(v)-J_Q(w)>0,
\end{eqnarray}
provided $l$ satisfies, beside  previous lower bounds, $l>l^*$ where $l^*$ is defined by the condition $Cre^{-\gamma l^*}=\frac{c^2}{128\mu^2}(1-\frac{4}{N^2})\epsilon^4$.
From the above part of the proof it follows that, if we set $l_\epsilon=2 l_M$ and if $\tilde{x}$ is such that $d(\tilde{x},\partial\Omega)\geq l_\epsilon$, then we can construct as before the set $Q$ and the map $w$ that coincides with $u$ outside $Q$ and satisfies (\ref{small-at-x2}) and (\ref{e-e1}) which contradicts the minimality of $u$. The proof is complete.
\end{proof}
Theorem \ref{teo} follows from Lemma \ref{de-exists}.

\section{The proof of Theorem \ref{teo2}}
\section{Basic lemmas}\label{basic-lemmas}
\begin{lemma}\label{v-property}

There exist positive constants $c,\;q^*$ such that
\begin{equation}\label{prima}
W^{\prime\prime}(q)\,\geq c^2,\text{ for } q\in (-q^*,q^*);
\end{equation}
\begin{equation}\label{seconda}
\begin{array}{l}
W(q)\,\geq\tilde{W}(q_0,q):= W(q_0)+W^\prime(q_0)(q-q_0) ,\\\medskip
\hspace{3 cm}\text{ for } (q_0,q)\in (0,q^*)\times(q_0,q^*]\cup(-q^*,0)\times[-q^*,q_0);
\end{array}
\end{equation}
\begin{equation}\label{terza}
\text{\rm sign}(q)W^\prime (q)\ge\text{\rm sign}(q)c^2q\geq 0 ,\quad\text{ for } q\in (-q^*, q^*);
\end{equation}
\end{lemma}
\begin{proof}
The inequality (\ref{prima}) follows immediately
from hypothesis (iii).
Now, the convexity of $W$ in $(-q^*, q^*)$ implies (\ref{seconda}).

To prove (\ref{terza}) note that, for $q\in (0,q^*),$
$$W^\prime (q)=\int_0^{q}W^{\prime\prime}(t) dt\ge c^2 q.$$
Analogously, for  $q\in (-q^*, 0),$
$$W^\prime (q)=-\int_{q}^0W^{\prime\prime}(t) dt\le c^2 q.$$
\end{proof}
By reducing the value of $q^*$ if necessary, we can also assume
\begin{equation}\label{vbar}
W(q^*\cdot \text{\rm sign}q)\le W(q)\le\overline{W},
\quad\text{ for } \vert q\vert\in[q^*,M_0],
\end{equation}
where $\overline{W}>0$ is a suitable constant. This follows from assumption (iii) and (\ref{bounds-udu}).

All the arguments that follow have a local character. Therefore, without loss of generality, in the remaining part of the proof we can assume that $\Omega$ is bounded.
\begin{lemma}\label{basic}
Assume  $R>0$  and $B_{x_0,R}\subset\Omega$ and
let $\varphi:B_{x_0,R}\rightarrow\R$ be the solution of
\begin{equation}\label{phi-comparison}
\left\{
\begin{array}{l}
\Delta \varphi = c^2\varphi, \text{ in } B_{x_0,R},\medskip\\
\varphi = \bar{q}, \text{ on } \partial B_{x_0,R},
\end{array}\right.
\end{equation}
where $\bar{q}\in(0,q^*]$.
Assume that $u\in W^{1,2}(\Omega)$ is a continuous map such that
\begin{eqnarray}
\vert u\vert\leq \bar{q}, \text{ for } x\in\overline{B}_{x_0,R}.
\end{eqnarray}
Then there exists a map $v\in W^{1,2}(\Omega)$ that satisfies:
\begin{eqnarray}\label{v-less-phi}
&&v=u, \text{ for } x\in\Omega\setminus B_{x_0,R},\\\nonumber
&&\vert v\vert\leq\varphi, \text{ for } x\in\overline{B}_{x_0,R}
\end{eqnarray}
and
\begin{eqnarray}\label{j-estimate}
\hskip.5cm J_\Omega(u)- J_\Omega(v)&=& J_{B_{x_0,R}}(u)-J_{B_{x_0,R}}(v)\\\nonumber  &\geq&\int_{B_{x_0,R}\cap\{\vert u\vert> \varphi\}}(W(u)-W(\varphi^u)-W^\prime(\varphi^u)(u-\varphi^u))dx,
\end{eqnarray}
where $\varphi^u=\text{\rm sign}(u)\varphi$.
\end{lemma}
\begin{proof}
Let $b>0$ be a number such that $b\leq\min_{x\in B_{x_0,R}}\varphi$. Since $u$ is continuous the set $A_b:=\{x\in B_{x_0,R}:u>b\}$ is open and there exists a function $\rho^+\in W^{1,2}(A_b)$ that minimizes the functional $J_{A_b}(p)=\int_{A_b}(\frac{1}{2}\vert\nabla p\vert^2+W(p))dx$ in the class of functions that satisfy the Dirichlet condition $p=u$ on $\partial A_b$. Since $\frac{\vert\rho^+\vert+\rho^+}{2}$ is also a minimizer we have $\rho^+\geq 0$. We also have $\rho^+\leq \bar{q}$. This follows from (\ref{terza}) and (\ref{vbar}) which imply that $\min\{\rho^+,\bar{q}\}$ is also a minimizer. The map $\rho^+$ satisfies the variational equation
\begin{eqnarray}\label{rho-variation0}
\int_{A_b}(\langle\nabla\rho^+,\nabla\eta\rangle+W^\prime(\rho^+)\eta)dx=0,
\end{eqnarray}
for all $\eta\in W_0^{1,2}(A_b)\cap L^\infty(A_b)$. In particular, if we define $A_b^*:=\{x\in A_b: \rho^+>\varphi\}$, we have
\begin{eqnarray}\label{rho-variation}
\int_{A_b^*}(\langle\nabla\rho^+,\nabla\eta\rangle+W^\prime(\rho^+)\eta)dx=0,
\end{eqnarray}
for all $\eta\in W_0^{1,2}(A_b)\cap L^\infty(A_b)$ that vanish on $A_b\setminus A_b^*$.

If we take $\eta=(\rho^+-\varphi)^+$ in (\ref{rho-variation}) and
use (\ref{terza}) we get
\begin{eqnarray}\label{rho-variation1}
\int_{A_b^*}(\langle\nabla\rho^+,\nabla(\rho^+-\varphi)\rangle+c^2\rho^+(\rho^+-\varphi))dx\leq 0,
\end{eqnarray}
This inequality and
\begin{eqnarray}\label{phi-variation1}
\int_{A_b^*}(\langle\nabla\varphi,\nabla(\rho^+-\varphi)\rangle+c^2\varphi(\rho^+-\varphi))dx=0,
\end{eqnarray}
that follows from (\ref{phi-comparison}), imply
\begin{eqnarray}\label{rho-variation2}
\int_{A_b^*}(\vert\nabla(\rho^+-\varphi)\vert^2+c^2(\rho^+-\varphi)^2)dx\leq 0.
\end{eqnarray}
That is $\mathcal{H}^n(A_b^*)=0$
which, together with $\rho^+\leq\varphi$ on $A_b\setminus A_b^*,$ shows that
\begin{eqnarray}\label{rho-min-phi}
\rho^+\leq\varphi, \text{ for } x\in A_b.
\end{eqnarray}
If we set $A_b^-:=\{x\in B_{x_0,R}:u<-b\}$ and $\rho^-\in W^{1,2}(A_b^-)$ is a minimizer of $J_{A_b^-}$ in the set of $W^{1,2}(A_b^-)$ maps that have the same trace of $u$ on $\partial A_b^-$, the argument above can be applied to $\rho^-$ to obtain
\begin{eqnarray}\label{rho-max-phi}
\rho^-\geq -\varphi, \text{ for } x\in  A_b^-.
\end{eqnarray}
Let $v\in W^{1,2}(\Omega)$ be the map defined by setting
\begin{eqnarray}\label{v-definition}
v=\left\{\begin{array}{l}
u,\text{ for } x\in\Omega\setminus A_b\cup A_b^-,\\
\min\{u,\rho^+\}, \text{ for } x\in A_b,\\
\max\{u,\rho^-\}, \text{ for } x\in A_b^-,
\end{array}\right.
\end{eqnarray}
This definition implies (\ref{v-less-phi}). Moreover we have
\begin{eqnarray}\label{diff-energy}
J_\Omega(u)- J_\Omega(v)&=& J_{A_b\cup A_b^-}(u)- J_{A_b\cup A_b^-}(v)\\\nonumber
 &=& J_{A_b\cap\{\rho^+<u\}}(u)- J_{A_b\cap\{\rho^+<u\}}(\rho^+)\\\nonumber
 &&+J_{A_b^-\cap\{\rho^->u\}}(u)- J_{A_b^-\cap\{\rho^->u\}}(\rho^-).
 \end{eqnarray}
 From (\ref{rho-variation0}) with $\eta=(u-\rho^+)^+$ it follows
 \begin{eqnarray}
\int_{A_b\cap\{\rho^+<u\}}\langle\nabla\rho^+,\nabla(u-\rho^+)\rangle &= &-\int_{A_b\cap\{\rho^+<u\}}W^\prime(\rho^+)(u-\rho^+)dx .
\end{eqnarray}
This and the identity
\begin{eqnarray}
\frac{1}{2}(\vert\nabla u\vert^2-\vert\nabla \rho^+\vert^2)&=&\frac{1}{2}\vert\nabla u-\nabla \rho^+\vert^2 +\langle\nabla\rho^+,\nabla( u-\rho^+)\rangle,
\end{eqnarray}
imply

 \begin{eqnarray}\label{j-difference+} J_{A_b\cap\{\rho^+<u\}}(u)- J_{A_b\cap\{\rho^+<u\}}(\rho^+)\hskip6.5cm\\\nonumber\\\nonumber =
\int_{A_b\cap\{\rho^+<u\}}(\frac{1}{2}\vert\nabla u-\nabla\rho^+\vert^2+\langle\nabla\rho^+,\nabla( u-\rho^+)\rangle +W(u)-W(\rho^+))dx\\\nonumber  \\\nonumber \geq\int_{A_b\cap\{\rho^+<u\}} (W(u)-W(\rho^+)-W^\prime(\rho^+)(u-\rho^+))dx\\\nonumber  \\\nonumber \geq\int_{A_b\cap\{\varphi<u\}} (W(u)-W(\varphi)-
W^\prime(\varphi)(u-\varphi))dx,
\end{eqnarray}
where we have used $A_b\cap\{\varphi<u\}\subset A_b\cap\{\rho^+<u\}$ and
the fact that the function $\tilde W(\cdot, u)$ defined in (\ref{seconda}) is increasing on
 $(0, u).$  In the same way one proves
 \begin{eqnarray}\label{j-difference-} J_{A_b^-\cap\{\rho^->u\}}(u)- J_{A_b^-\cap\{\rho^->u\}}(\rho^-)\hskip6.5cm\\\nonumber\\\nonumber   \geq\int_{A_b^-\cap\{-\varphi>u\}} (W(u)-W(-\varphi)-W^\prime(-\varphi)(u+\varphi))dx.
\end{eqnarray}
This inequality and (\ref{j-difference+}) imply (\ref{j-estimate}).
\end{proof}
Given $\bar{q}\in(0,q^*)$ define
\begin{equation}\label{rbar}
\overline{R}=\frac{q^*-\bar{q}}{M_0}
\end{equation}
 where $M_0$ is the constant in (\ref{bounds-udu}).
For later reference we quote
\begin{corollary}\label{energy-corollary}
Let $\lambda>0$ be fixed, assume that $R>\overline{R}$ is such that $B_{x_0,R+\lambda/2}\subset \Omega$ and let $u\in W^{1,2}(\Omega)$ a continuous map that satisfies the condition
\begin{eqnarray}
\vert u\vert\leq\bar{q}, \text{ for } x\in \partial B_{x_0,R+\lambda/2}.
\end{eqnarray}
 Then, there exist a constant $k>0$ independent of  $R>\overline{R}$ and a map $v\in W^{1,2}(\Omega)$ such that
 \begin{eqnarray}\label{diff-energy-corollary}
&& v=u,\;\text{ on }\;\Omega\setminus B_{x_0,R+\lambda/2},\\\nonumber\\\nonumber
&& J_{\Omega} (u)-J_{\Omega} (v)=J_{B_{x_0,R+\lambda/2}}(u)- J_{B_{x_0,R+\lambda/2}}(v)
\geq k\mathcal{H}^n(A_{\bar{q}}\cap B_{x_0,R}),
\end{eqnarray}
where $A_{\bar{q}}:=\{x\in\Omega:\vert u\vert>\bar{q}\}$.
\end{corollary}
\begin{proof}

Let $\hat{u}\in W^{1,2}(\Omega)$ be defined by
\begin{eqnarray}
\hat{u}=\left\{\begin{array}{l}
\bar{q},\;\;\ \text{ on }  B_{x_0,R+\lambda/2}\cap \{u>\bar{q}\},\\
-\bar{q},\;\;\text{ on }  B_{x_0,R+\lambda/2}\cap \{u<-\bar{q}\},
\\
u,\;\;\text{ otherwise }.
\end{array}\right.
\end{eqnarray}
Then, using also (\ref{vbar}), we have
 \begin{eqnarray}\label{diff-energy5}
&&J_{\Omega}(u)- J_{\Omega}(\hat{u})
 = \int_{ B_{x_0,R+\lambda/2}\cap \{u>\bar{q}\}}(\frac{1}{2}\vert\nabla u\vert^2+W(u)-W(\bar{q}))dx \\\nonumber &&\hskip2.5cm
+\int_{B_{x_0,R+\lambda/2}\cap \{u<-\bar{q}\}}(\frac{1}{2}\vert\nabla u\vert^2+W(u)-W(-\bar{q}))dx\geq 0.
\end{eqnarray}
The map $\hat{u}$ satisfies the assumptions of Lemma \ref{basic}. Therefore if we let $v$ be the map associated to $\hat{u}$ by Lemma \ref{basic} (for $R+\lambda/2$), from (\ref{diff-energy5}) and (\ref{j-estimate}) we obtain
\begin{equation}\label{j-estimate1}
\begin{split}
 J_{\Omega}(u)- J_{\Omega}(v)&= J_{B_{x_0,R+\lambda/2}}(u)-J_{B_{x_0,R+\lambda/2}}(v)\\&\ge
 J_{B_{x_0,R+\lambda/2}}(\hat u)-J_{B_{x_0,R+\lambda/2}}(v)
\\ &\geq\int_{A_{\bar{q}}\cap B_{x_0,R+\lambda/2}} (W(\hat u)-W(\varphi^{\hat u})-W^\prime(\varphi^{\hat u})(\hat u-\varphi^{\hat u}))dx,
\end{split}
\end{equation}
where we have also used $A_{\bar{q}}\cap B_{x_0,R+\lambda/2}\subset B_{x_0,R+\lambda/2}\cap\{\vert \hat{u}\vert>\varphi\}$.

We have $\varphi(x)=\phi(\vert x-x_0\vert,R+\lambda/2)$ with $\phi(\cdot,R+\lambda/2):[0,R+\lambda/2]\rightarrow\R$ a positive function which is strictly increasing in $(0,R+\lambda/2]$. Moreover we have $\phi(R+\lambda/2,R+\lambda/2)=\bar{q}$ and
\begin{eqnarray}\label{phi-l}
R_1<R_2\;\;\Rightarrow\;\;\phi(R_1-\lambda,R_1)>\phi(R_2-\lambda,R_2).
\end{eqnarray}
Note that $x\in B_{x_0,R}$ implies $\varphi(x)\leq \phi(R,R+\lambda/2)$.
Therefore for $x$ in the subset of $A_{\bar{q}}\cap B_{x_0,R}$ where $u>\varphi$ we have
 \begin{eqnarray}\label{diff-potential}
 W(\bar{q})-W(\varphi)-W^\prime(\varphi)(\bar{q}-\varphi)
 =\int_{\varphi}^{\bar{q}}(W^\prime(q)-W^\prime(\varphi))dq\hspace{1 cm}\\\nonumber
 \geq c^2\int_{\varphi}^{\bar{q}}(q-\varphi)dq=\frac{1}{2}c^2(\bar{q}-\varphi)^2\geq
 \frac{1}{2}c^2(\phi(R+\lambda/2,R+\lambda/2)-\phi(R,R+\lambda/2))^2,
\end{eqnarray}
where we have also used $(\ref{prima})$. In a similar way we derive the estimate
 \begin{eqnarray}\label{diff-potential-}
 W(-\bar{q})-W(-\varphi)-W^\prime(-\varphi)(-\bar{q}+\varphi)
 =\int_{-\varphi}^{-\bar{q}}(W^\prime(q)-W^\prime(-\varphi))dq\\\nonumber
 \geq -c^2\int_{-\bar{q}}^{-\varphi}(q+\varphi)dq=\frac{1}{2}c^2(\bar{q}-\varphi)^2 \geq
 \frac{1}{2}c^2(\phi(R+\lambda/2,R+\lambda/2)-\phi(R,R+\lambda/2))^2,
\end{eqnarray}
valid in the subset of $A_{\bar{q}}\cap B_{x_0,R}$ where $u<-\varphi$.
The corollary follows from this and (\ref{diff-potential}), from (\ref{j-estimate1}) and from the fact that, by (\ref{phi-l}), the last expression in (\ref{diff-potential}) and (\ref{diff-potential-}) is increasing with $R$. Therefore we can assume
\begin{eqnarray}
k=\frac{1}{2}c^2(\phi(\overline{R}+\lambda/2,\overline{R}+\lambda/2)-\phi(\overline{R},\overline{R}+\lambda/2))^2.
\end{eqnarray}
\end{proof}

\begin{lemma}\label{cutting-lemma}
Let  $u\in W^{1,2}(\Omega)$ be a local minimizer as in Theorem \ref{teo2}.
Let $\lambda>0$ be fixed and assume that $B_{x_0,R+\lambda}\subset\Omega$ for some $R>\overline{R}$. Assume
\begin{equation}\label{serve}
A_{\bar{q}}\cap B_{x_0,R}\neq\varnothing\;,
\end{equation}
and let $S=A_{\bar{q}}\cap (B_{x_0,R+\lambda}\setminus{\overline{B_{x_0,R}}})$. Then, there exist a constant $K>0$ independent of  $R>\overline{R}$ and a continuous map $v\in W^{1,2}(\Omega) $ that satisfies
\begin{equation}\label{v-def-1}
\left\{\begin{array}{l}
v = u, \text{ for } x\in\Omega\setminus S,\\
\text{\rm sign}(u)v > \bar{q}, \text{ for } x\in A_{\bar{q}}\cap B_{x_0,R+\frac{\lambda}{2}},\\
\text{\rm sign}(u)v = \bar{q}, \text{ for } x\in \partial(A_{\bar{q}}\cap B_{x_0,R+\frac{\lambda}{2}}),
\end{array}\right.
\end{equation}
and
\begin{equation}\label{secondaLemma2.4}
J_{\Omega}(v)-J_{\Omega}(u)=J_S(v)-J_S(u)\leq K\mathcal{H}^n(S).
\end{equation}
\end{lemma}
\begin{proof}
 From Corollary \ref{energy-corollary} and the minimality of $u$ we necessarily have $A_{\bar{q}}\cap \partial B_{x_0,R+\frac{\lambda}{2}}\neq\emptyset .$
Indeed, if on the contrary $A_{\bar{q}}\cap \partial B_{x_0,R+\frac{\lambda}{2}}=\emptyset,$ then $\vert u\vert\le \bar{q}$ on $\partial B_{x_0,R+\frac{\lambda}{2}}.$
Therefore, applying Corollary  \ref{energy-corollary} to $u$ on $B_{x_0,R+\frac{\lambda}{2}},$ we could find $v$ satisfying
$$
J_{\Omega} (u)-J_{\Omega} (v)=J_{B_{x_0,R+\frac{\lambda}2}}(u)- J_{B_{x_0,R+\frac{\lambda}2}}(v)
\geq k\mathcal{H}^n(A_{\bar{q}}\cap B_{x_0,R}).$$
From (\ref{serve}), this is
in contradiction with the minimality of $u.$

Let $v\in W^{1,2}(\Omega)$ be defined by $v=u$ for $x\not\in S$ and by
\begin{eqnarray}\label{v-def-2}
v=(1-\vert 1-2\frac{r-R}{\lambda}\vert)\text{\rm sign}(u)\bar{q}+\vert 1-2\frac{r-R}{\lambda}\vert u, \text{ for } x\in S,
\end{eqnarray}
where $r=\vert x-x_0\vert$.
From this definition and (\ref{bounds-udu}) it follows
\begin{eqnarray}\label{q-bound-ins}
\bar{q}<\text{\rm sign}(u)v\leq \vert u\vert\leq M_0, \text{ for }&& x\in S\setminus\partial B_{x_0,R+\frac{\lambda}{2}},\\\nonumber
v=\text{\rm sign}(u)\bar{q},\quad\quad \text{ for }&& x\in S\cap\partial B_{x_0,R+\frac{\lambda}{2}}.
\end{eqnarray}
Moreover, it is easy to verify that $v=u$ on $\partial S.$ Then, $v$ is continuous and satisfies $(\ref{v-def-1}).$
 From (\ref{v-def-2}) we also obtain
\begin{eqnarray}\label{nabla-qv}
\nabla v= \Big \vert 1-2\frac{r-R}{\lambda}\Big\vert
\nabla u+\frac{2}{\lambda}(u-\text{\rm sign}(u)\bar{q})\nu, \text{ for } x\in S,
\end{eqnarray}
where $\nu=-{\rm sign}(1-2\frac{r-R}{\lambda})\frac{x-x_0}{r}$.
From (\ref{nabla-qv}), (\ref{q-bound-ins}) and (\ref{bounds-udu}) it follows
\begin{eqnarray}\label{nabla-qv-1}
\frac{1}{2}(\vert\nabla v\vert^2-\vert\nabla u\vert^2)+W(v)-W(u)\hskip2.5cm\\\nonumber
\leq \frac{1}{2}(\vert \nabla u\vert+\frac{2}{\lambda}\vert u-\text{\rm sign}(u)\bar{q}\vert )^2+\overline{W}-W(\text{\rm sign}(u)\bar{q})
\hskip1cm\\\nonumber
\leq \frac{1}{2}(  M_0+\frac{2}{\lambda}(M_0-\bar{q}))^2+\overline{W}, \text{ for } x\in S,
\end{eqnarray}
where $\overline{W}$ is the constant in Lemma \ref{v-property}.
The estimate (\ref{nabla-qv-1}) concludes the proof with $K$ given by the last expression in (\ref{nabla-qv-1}).
\end{proof}

\begin{proposition}\label{r0-existence}
Let $\bar{q}\in(0,q^*)$, $\lambda>0$ and
${\overline R}=\frac {q^*-\bar{q}}{M_0}$ as before.
There exists $j_m\in\nat$ such that, if $R_0=\overline {R}+(j_m+1)\lambda$, then a local minimizer $u$  satisfies
\begin{eqnarray}\label{below}
x\in\Omega,\;\; d(x,\partial\Omega)\;\geq R_0\quad\Rightarrow\quad \vert u\vert\;<\; q^*.
\end{eqnarray}
Moreover the number $j_m$ depends only  on $\bar{q},\;\lambda\;$ and the constants $k,\;K$ in Corollary $\ref{energy-corollary}$ and Lemma $\ref{cutting-lemma}.$
\end{proposition}
\begin{proof}
Suppose that $\vert u(x_0)\vert\geq q^*$ for some $x_0\in\Omega.$
Then, from (\ref{bounds-udu}),
$$\vert u(x)\vert  >\bar{q}, \quad \forall x\in
B_{x_0,{\overline R}}.$$
Therefore, if $d(x_0, \partial\Omega )\ge {\overline R},$ (\ref{bounds-udu}) implies
\begin{eqnarray}
\mathcal{H}^n(A_{\bar{q}}\cap B_{x_0,\overline {R}})=\mathcal{H}^n( B_{x_0,\overline {R}}):=\sigma_0.
\end{eqnarray}
Now, set
\begin{eqnarray}
\sigma_j:=\mathcal{H}^n(A_{\bar{q}}\cap B_{x_0,\overline {R}+j\lambda}),
\end{eqnarray}
for each $j\in\nat$ such that $d(x_0, \partial\Omega)\ge \overline {R}+(j+1)\lambda$.

Let $v_j^1, v_j^2\in W^{1,2}(\Omega)$ be the maps defined as follows:
\begin{description}
	\item[]$v_j^1$ is the map $v$ defined in Lemma \ref{cutting-lemma} for $B_{x_0,R+\lambda}$ with $R=\overline {R}+j\lambda$.
	\item[]$v_j^2$ is the map $v$ given by Corollary \ref{energy-corollary} when $u=v_j^1$ and $R=\overline {R}+j\lambda$.
\end{description}
From these definitions, Corollary \ref{energy-corollary} and Lemma \ref{cutting-lemma}, we deduce
\begin{eqnarray}
J_\Omega(u)-J_\Omega(v_j^1)&\geq& -K (\sigma_{j+1}-\sigma_j),\\\nonumber
J_\Omega(v_j^1)-J_\Omega(v_j^2)&\geq& k\mathcal{H}^n(A_{\bar{q}}\cap
B_{x_0, \overline{R}+j\lambda})= k \sigma_j.
\end{eqnarray}
By adding these inequalities and using the minimality of $u$ we obtain
\begin{eqnarray}\label{first-sigma-rel}
0\geq J_\Omega(u)-J_\Omega(v_j^2)\geq k \sigma_{j} -K (\sigma_{j+1}-\sigma_j)
\end{eqnarray}
and therefore,
\begin{eqnarray}\label{second-sigma-rel}
\Big (1+\frac{k}{K}\Big )\sigma_{j-1}\leq\sigma_j&\leq&\frac{K}{k}(\sigma_{j+1}-\sigma_j) ,\;\; j\in\nat ,\\\nonumber
\Rightarrow \Big (1+\frac{k}{K}\Big )^{j-1}\sigma_0\leq\sigma_j&\leq&\omega\frac{K}{k}\Big((\overline{R}+(j+1)\lambda)^n-(\overline{R}+j\lambda)^n\Big) ,\;\; j\in\nat .
\end{eqnarray}
where $\omega$ is the measure of the unit ball in $\R^n$.
For $j$ sufficiently large the last inequality is not satisfied and this contradicts the minimality of $u.$
We denote $j_m$ the minimum value of $j$ such that (\ref{second-sigma-rel}) is violated.
Then, (\ref{below}) follows with $R_0=\overline{R}+(j_m+1)\lambda$.
\end{proof}
The existence of the map $(0,q^*]\ni q\rightarrow R(q)$ follows from the fact that all the above arguments can be repeated with a generic $q\in(0,q^*)$ in place of $q^*$. We can obviously assume that $R(q)$ is decreasing and, by modifying it if necessary, we can also assume that it is strictly decreasing and continuous.

For completing the proof of Theorem \ref{teo2}  it remains to prove the estimate (\ref{exp-bound}). Proposition \ref{r0-existence} and in particular (\ref{below}) imply that we can apply Lemma \ref{basic} to $u$ and the ball $B_{x,R}$ for each $x\in \Omega$ such that $d(x,\partial\Omega)= R_0+R$ with $R\geq R_0$. Therefore we obtain
\begin{eqnarray}\label{1}
\vert u(x)\vert\leq\phi(0,R).
\end{eqnarray}
We also have (see \cite{flp}) that
\begin{eqnarray}\label{2}
\phi(0,R)\leq q^* e^{-k_0 R}=q^* e^{k_0 R_0}e^{-k_0d(x,\partial\Omega)},
\end{eqnarray}
for some $k_0>0$ independent of $R\in[\overline{R},+\infty)$.
From (\ref{1}), (\ref{2}) we obtain
\begin{eqnarray}\label{3}
\vert u(x)\vert\leq q(R)\leq K_0e^{-k_0 d(x,\partial\Omega)},\;\text{ for }\;d(x,\partial\Omega)\geq 2R_0,
\end{eqnarray}
  This concludes the proof of Theorem \ref{teo2}.

\bibliographystyle{plain}

\end{document}